\documentclass[reqno, 12pt,a4letter]{amsart}
\usepackage{amsmath,amsxtra, amssymb, latexsym, amscd,amsthm}
\usepackage{graphicx,color}
\usepackage[active]{srcltx}
\usepackage[utf8]{inputenc}
\usepackage[mathscr]{euscript}
\usepackage{mathrsfs}
\usepackage[english]{babel}
\usepackage{color}
\newtheorem {theorem}{Theorem}[section]

\newtheorem {proposition}{Proposition}[section]

\newtheorem {example}{Example}[section]
\newtheorem {definition}{Definition}[section]
\newtheorem {remark}{Remark}[section]

\newcommand{\conv}{\mbox{\rm conv}\,}

\newcommand{\cl}{\mbox{\rm cl}\,}

\def\ar{a\kern-.370em\raise.16ex\hbox{\char95\kern-0.53ex\char'47}\kern.05em}
\def\ees{{\accent"5E e}\kern-.385em\raise.2ex\hbox{\char'23}\kern-.08em}
\def\eex{{\accent"5E e}\kern-.470em\raise.3ex\hbox{\char'176}}
\def\AR{A\kern-.46em\raise.80ex\hbox{\char95\kern-0.53ex\char'47}\kern.13em}
\def\EES{{\accent"5E E}\kern-.5em\raise.8ex\hbox{\char'23 }}
\def\EEX{{\accent"5E E}\kern-.60em\raise.9ex\hbox{\char'176}\kern.1em}
\def\ow{o\kern-.42em\raise.82ex\hbox{
		\vrule width .12em height .0ex depth .075ex \kern-0.16em \char'56}\kern-.07em}
\def\OW{O\kern-.460em\raise1.36ex\hbox{
		\vrule width .13em height .0ex depth .075ex \kern-0.16em \char'56}\kern-.07em}
\def\UW{U\kern-.42em\raise1.36ex\hbox{
		\vrule width .13em height .0ex depth .075ex \kern-0.16em \char'56}\kern-.07em}
\def\DD{D\kern-.7em\raise0.4ex\hbox{\char '55}\kern.33em}

\title[]{Strong Second-Order Karush--Kuhn--Tucker Optimality Conditions for Vector Optimization}

\author{NGUYEN QUANG HUY}

\address{Department of Mathematics, Hanoi Pedagogical University 2, Xuan Hoa, Phuc Yen, Vinh Phuc, Vietnam}

\email{huyngq308@gmail.com}

\author{DO SANG KIM$^\dagger$}
\address{Department of Applied Mathematics, Pukyong National University, Busan 48513, Korea}
\email{dskim@pknu.ac.kr}

\author{NGUYEN VAN TUYEN$^\ddagger$}
\address{Department of Mathematics, Hanoi Pedagogical University 2, Xuan Hoa, Phuc Yen, Vinh Phuc, Vietnam}

\email{tuyensp2@yahoo.com; nguyenvantuyen83@hpu2.edu.vn}

\thanks{$^\dagger$The research of the second author was supported by the National Research Foundation of Korea Grant funded by the Korean Government (NRF-2016R1A2B4011589)}

\thanks{$^{\ddagger}$The research of the third author was supported by the National Foundation for Science and Technology Development (NAFOSTED), Vietnam grant 101.01-2014.39}



\date{ \today}

\begin{document}
\maketitle

\begin{abstract}
	In the present paper, we focus on the vector optimization problems  with inequality constraints, where objective functions and constrained functions are vector-valued functions with $C^{1,1}$ components defined on $\mathbb{R}^n$. By using the second-order symmetric subdifferential and the second-order tangent set, we propose two types of second-order regularity conditions in the sense of Abadie. Then we establish some strong second-order Karush--Kuhn--Tucker   necessary optimality conditions for Geoffrion properly efficient solutions of the considered problem. Examples are given to illustrate the obtained results.

\end{abstract}

\section{Introduction}
In this paper, we are interested in second-order optimality conditions for the following constrained vector optimization  problem
\begin{align*}
& \text{min}\, f(x)\label{problem} \tag{VP}
\\
&\text{s. t.}\ \ x\in Q_0:=\{x\in\mathbb{R}^n\,:\, g(x)\leqq 0\},
\end{align*}
where $f:=(f_i)$, $i\in I:=\{1, \ldots, l\}$, and $g:=(g_j)$, $j\in J:=\{1, \ldots, m\}$  are vector-valued functions with $C^{1,1}$ components defined on $\mathbb{R}^n$. Recall that a real-valued function $\varphi$ is said to be a $C^{1,1}$ function if it is Fr\'echet differentiable with a locally Lipschitz gradient $\nabla \varphi(\,\cdot\,)$ on $\mathbb{R}^n$; see \cite{hien84} for more details.

It is well-known that if $x^0\in Q_0$ is an efficient solution to \eqref{problem}, then there exist Lagrange multipliers $(\lambda, \mu)\in \mathbb{R}^l\times\mathbb{R}^m$ satisfying
\begin{align}
&\sum_{i=1}^l\lambda_i\nabla f_i(x^0)+\sum_{j=1}^m\mu_j\nabla g_j(x^0)=0,\label{equa_intro:1}
\\
&\mu=(\mu_1, \ldots, \mu_m)\geqq 0, \mu_jg_j(x^0)=0,\label{equa_intro:2}
\\
&\lambda=(\lambda_1, \ldots, \lambda_l)\geqq 0, (\lambda, \mu)\neq 0;\label{equa_intro:3}
\end{align}
see \cite[Theorem 7.4]{Jahn04}. The conditions \eqref{equa_intro:1}--\eqref{equa_intro:3} are called the first-order F.-John necessary optimality conditions. If $\lambda$ is nonzero, then this conditions are called the Karush--Kuhn--Tucker $(KKT)$  optimality  conditions.  $(KKT)$ optimality conditions are one of the most important results in optimization theory. In vector optimization problems, there are two kinds of $(KKT)$ optimality conditions. When all Lagrange multipliers corresponding to the objective functions are positive, we say that {\em strong first-order Karush--Kuhn--Tucker $(SFKKT)$} conditions hold. On the other hand, when at least one of Lagrange multipliers, corresponding to the objective functions, is positive, we say that {\em weak first-order Karush--Kuhn--Tucker $(WFKKT)$} conditions hold for the problem.  In this case, some Lagrange multipliers corresponding to the components of the vector objective function may be zero. This means that some components of the vector-valued objective function have no role in the necessary conditions. To avoid this situation and to obtain positive Lagrange multipliers associated with each of the objective functions, the problem has to fulfill some assumptions. These assumptions are called constraint qualifications $(CQ)$ when they have to be fulfilled by the constraints of the problem, and they are called regularity conditions $(RC)$ when they have to be fulfilled by both the objectives and the constraints of the problem.

The optimality conditions for vector problems, which use similar $(CQ)$'s as those used for single-objective problems do not ensure $(SFKKT)$ conditions; see \cite{Kuhn52,Singh87}. In 1994, Maeda \cite{Maeda94} was the first to introduce a Generalized Guignard regularity condition and established   $(SFKKT)$ necessary conditions for differentiable problems. Later on, Preda and Chi{\c{t}}escu \cite{Preda99}  derived  $(SFKKT)$ necessary conditions for  efficient solutions  of   semidifferentiable vector optimization problems.  In the recent years, there are many works dealing with  $(SFKKT)$ necessary optimality conditions  for smooth and nonsmooth vector optimization problems; see, for example, \cite{Rizvi12,Giorgi 09,Golestani13,Chuong14}. 

One of the first investigations to obtain second-order (KKT)   optimality conditions for smooth vector optimization problems was carried out by Wang \cite{Wang91}. Then, Bigi and Castellani \cite{Bigi00,Bigi04} obtained some weak second-order (KKT) optimality conditions by introducing some types of the second-order regularity conditions. On the line of their work, many authors have derived weak second-order (KKT) necessary conditions for efficiency in vector optimization problems both for smooth and nonsmooth cases; see, for example, \cite{novo09,novo03,Khanh16,Tuan15,Guerraggio01,Tuan16,Aghezzaf99,Aghezzaf07}. However, to the best of our knowledge, there are only a few works considering  strong second-order Karush--Kuhn--Tucker $(SSKKT)$ necessary optimality conditions. In \cite{Maeda04}, Maeda was the first to propose a Abadie second-order regularity condition $(ASORC)$ and established  $(SSKKT)$ necessary conditions in terms of generalized second-order directional derivatives for $C^{1,1}$ vector optimization problems. Recently, Kim and Tuyen \cite{Kim17} obtained some $(SSKKT)$ necessary optimality conditions for Geoffrion properly efficient solutions of $C^2$ vector optimization problems under the so-called  generalized  Abadie second-order regularity condition $(GASORC)$. The $(GASORC)$ was first introduced by Rizvi and Nasser in \cite{Rizvi06}. As shown in \cite{Rizvi06}, this condition is weaker than the condition $(ASORC)$.

Our aim is to extend \cite[Theorem 3.2]{Kim17} to a larger class of vector optimization problems. By using the second-order symmetric subdifferential, which was introduced in \cite{Huy16}, we propose two types of second-order regularity conditions in the sense of Abadie for \eqref{problem}. These regularity conditions generalize corresponding regularity conditions in \cite{Rizvi06} to $C^{1,1}$ vector optimization problems. Then we establish some $(SSKKT)$ necessary optimality conditions in terms of second-order symmetric subdifferentials  for Geoffrion properly efficient solutions of \eqref{problem}. As shown in \cite{Huy16}, the second-order symmetric subdifferential may be strictly smaller than the Clarke  subdifferential, and  has some nice properties. In particular, every $C^{1,1}$ function has Taylor expansion in terms of its second-order symmetric subdifferential. This property plays an important role in
our paper.

The rest of the paper is organized as follows.  In Section \ref{Preliminaries}, we recall some basic  definitions and preliminaries from variational analysis, which are widely used in the  sequel. Section \ref{Regularity} is devoted to investigate  second-order regularity conditions in the sense of Abadie for vector optimization problems. In Section \ref{Strong_KT_section}, we establish some $(SSKKT)$ optimality conditions for Geoffrion properly efficient solutions of \eqref{problem}. Section \ref{Conclusion}  draws some conclusions.  
\section{Preliminaries}
\label{Preliminaries}
In this section, we recall some basic  definitions and preliminaries from variational analysis, which are widely used in the  sequel. Let $\Omega$ be a subset in $\mathbb{R}^n$. The  {\it closure}, {\it convex hull} and {\it conic hull} of $\Omega$ are denoted, respectively, by $\mbox{cl}\,\Omega$,
$\conv\,\Omega$ and $\mbox{cone}\,\Omega$. 

\begin{definition}{\rm
		Given $\bar x\in  \mbox{cl}\Omega$. The  {\em limiting normal cone} or the {\em Mordukhovich normal cone}  of $\Omega$ at $\bar x$  is the set
		\begin{equation*}
		N(\bar x;\Omega):=\{z^*\in \mathbb{R}^n:\exists
		x^k\stackrel{\Omega}\longrightarrow\bar x, \epsilon_k\downarrow 0, z^*_k\to z^*,
		z^*_k\in {\widehat N_{\epsilon
				_k}}(x^k; \Omega),\,\,\forall k \in\mathbb{N}\},
		\end{equation*}
		where
		\begin{equation*}
		\Hat N_\epsilon  (x; \Omega):= \left\{ {z^*  \in {\mathbb{R}^n} \;:\;\limsup_{u\overset{\Omega} \rightarrow x}
			\frac{{\langle z^* , u - x\rangle }}{{\parallel u - x\parallel }} \leq \epsilon } \right\}
		\end{equation*}
		is the set of  {\em $\epsilon$-normals} of $\Omega$  at $x$ and the notation $u\xrightarrow {{\Omega}} x$ means that $u \rightarrow x$ and $u \in \Omega$.
	}
\end{definition}

\begin{definition}{\rm Let $C$ be a nonempty subset of $\mathbb{R}^n$, $x^0\in C$ and $u\in\mathbb{R}^n$. 
		\begin{enumerate}
			\item [(i)] The {\em tangent cone} to $C$ at $x^0\in C$ is the set defined by 
			$$T(C; x^0):=\{d\in\mathbb{R}^n\,:\,\exists t_k\,\downarrow 0, \exists d^k\to d, x^0+t_kd^k\in C, \ \ \forall k\in \mathbb{N}\}.$$	
			\item [(ii)]	The {\em second-order tangent set} to $C$ at $x^0$ with respect to the direction $u$ is the set defined by 
			$$T^2(C; x^0, u):=\left\{v\in\mathbb{R}^n:\exists t_k\,\downarrow 0, v^k\to v, x^0+t_ku+\frac12t_k^2v^k\in C, \, \forall k\in\mathbb{N}\right\}.$$
		\end{enumerate}		
	}	
\end{definition}
From the definition, we have $T^2(C; x^0, 0)=T(C; x^0)$.  Clearly, $T(\cdot\,; x^0)$ and $T^2(\cdot\,; x^0, u)$ are isotone, i.e., if $C^1\subset C^2$, then
\begin{align*}
T(C^1; x^0)&\subset T(C^2; x^0),
\\
T^2(C^1; x^0, u)&\subset  T^2(C^2; x^0, u).
\end{align*}

It is well-known that $T(C; x^0)$ is a nonempty closed cone. For each $u\in \mathbb{R}^n$, the set $T^2(C; x^0, u)$ is closed, but may be empty. However, we see that the set $T^2(C; x^0, 0)=T(C; x^0)$ is always nonempty.  If $C$ is convex, then
$$T(C; x^0)=\mbox{cl}\, \left\{d\,:\,d=\beta (x-x^0), x\in C, \beta\geq 0\right\},$$
and for each $u\in T(C; x^0)$ one has
$$
T^2(C; x^0, u)\subset \mbox{cl cone}[ \mbox{ cone}(C-x^0)-u].
$$
Moreover, if $C$ is a polyhedral convex set, then $ T^2(C; x^0, u)= T({T(C; x^0)};u).$

Let  $\varphi\colon \mathbb{R}^n \to \overline{\mathbb{R}}$ be an {\it extended-real-valued function}. The {\it  epigraph}, {\it  hypergraph} and  {\it domain} of $\varphi$ are denoted, respectively, by
\begin{align*}
\mbox{epi }\varphi&:=\{(x, \alpha)\in\mathbb{R}^n\times\mathbb{R} \,:\,  \alpha\geq \varphi(x) \},
\\
\mbox{hypo }\varphi &:=\{(x, \alpha)\in\mathbb{R}^n\times\mathbb{R} \,:\, \alpha\leq \varphi(x) \}, 
\\
\mbox{dom }\varphi &:= \{x\in \mathbb{R}^n \,:\, |\varphi(x)|<+\infty \}.
\end{align*}

\begin{definition}{\rm
		Let $\bar x\in \mbox{dom }\varphi$. The set
		\begin{equation*} 
		\partial \varphi (\bar x):=\{x^*\in \mathbb{R}^n \,:\, (x^*, -1)\in N((\bar x, \varphi (\bar x)); \mbox{epi }\varphi )\}
		\end{equation*}
		is called the {\it limiting subdifferential}, or the {\em Mordukhovich subdifferential}, of $\varphi$ at $\bar x$. If $\bar x\notin \mbox{dom }\varphi$, then we put $\partial \varphi (\bar x)=\emptyset$.
	}
\end{definition}

\begin{definition}{\rm (see \cite[p. 84]{Mor06a})
		Given $\bar x\in \mbox{dom }\varphi.$ The sets
		\begin{align*}
		\partial ^+ \varphi (\bar x)&:=\{x^*\in \mathbb{R}^n \,:\, (-x^*, 1)\in N((\bar x, \varphi (\bar x)); \mbox{hypo }\varphi )\},
		\\
		\partial _S \varphi (\bar x)&:=\partial \varphi (\bar x)\cup \partial^+ \varphi (\bar x), 
		\\
		\partial _C \varphi (\bar x)&:=\cl\conv \partial _S \varphi (\bar x)
		\end{align*}
		are called the {\em upper subdifferential}, the {\em symmetric subdifferential} and the {\em Clarke subdifferential} of $\varphi$ at $\bar x$, respectively.
	}
\end{definition}

We note here that 
$$\partial \varphi (\bar x) \subseteq \partial _S \varphi (\bar x) \subseteq \partial _C \varphi (\bar x),$$
and both inclusions may be strict, see \cite[pp. 92--93]{Mor06a}.

Let $D$ be an open subset of $\mathbb{R}^n$. We denote by  $C^{1, 1}(D)$ the class of all real-valued functions $\varphi$, which are Fr\'echet differentiable on $D$, and whose gradient mapping $\nabla \varphi(\cdot)$ is  locally  Lipschitz on $D$.

\begin{definition}{\rm (see \cite[Definition 2.6]{Huy16})
		Let $\varphi\in C^{1, 1}(D)$ and $\bar x\in D$. The {\em second-order  symmetric subdifferential} of $\varphi$ at $\bar x$ is a multifunction $$\partial_S^2\varphi(\bar x)\colon \mathbb{R}^n \rightrightarrows \mathbb{R}^n$$  defined by
		\begin{equation*}
		\partial_S^2\varphi (\bar x)(v):=\partial_S\langle v, \nabla \varphi \rangle (\bar x)=\partial\langle v, \nabla \varphi \rangle (\bar x) \cup \partial^+\langle v, \nabla \varphi \rangle (\bar x), \ \ \forall v\in \mathbb{R}^n.
		\end{equation*}
	}
\end{definition}

We now summarize some properties of  the second-order  symmetric subdifferential that will be needed in this paper.
\begin{proposition}\label{second-order_subdiff}{\rm (see \cite[Proposition 2.3]{Huy16})}
	Let $\varphi\in C^{1,1}(D)$ and $\bar x\in D$. The following assertions hold:
	\begin{enumerate}
		\item  [\rm(i)] For any $\lambda\in \mathbb{R}$ and $v\in \mathbb{R}^n$, we have
		$$
		\partial_S^2\varphi (\bar x)(\lambda v)=\partial_S^2(\lambda \varphi) (\bar x)(v)=\lambda\partial_S^2\varphi (\bar x)(v).
		$$
		\item  [\rm(ii)] For any $v\in\mathbb{R}^n$, $\partial_S^2\varphi (\bar x)(v)$ is a nonempty compact set in $\mathbb{R}^n$.
		\item  [\rm(iii)] For any $v\in\mathbb{R}^n$ the mapping $x\mapsto \partial_S^2\varphi (x)(v)$ is locally bounded. Moreover, if $x_k\to \bar x$, $x^*_k\to x^*$ and $x^*_k\in \partial_S^2\varphi ( x_k)(v)$ for all $k\in\mathbb{N}$, then $x^*\in \partial_S^2\varphi (\bar x)(v)$.
	\end{enumerate}
\end{proposition}

\begin{theorem}[Taylor's formula]{\rm (see \cite[Corollary 2.1]{Huy16})} Let $\varphi \in C^{1, 1}(\mathbb{R}^n)$. Then, for every $a, b\in \mathbb{R}^n$ there exists $z^*\in \partial \varphi_S^2(\xi)(b-a)$, where $\xi\in\,(a, b)$, such that
	$$
	\varphi(b)-\varphi(a)-\langle \nabla \varphi (a), b-a \rangle =\dfrac12\langle z^*, b-a \rangle.
	$$
\end{theorem}

\section{Second-Order Abadie Regularity Conditions}
\label{Regularity}
In this section, we propose some types of second-order regularity conditions in the sense of Abadie for  vector optimization problems, investigate some relations with the regularity conditions in \cite{Maeda04,Rizvi06}, and give some conditions which assure  that these regularity conditions hold true. 

We first recall notations and definitions which will be used in the sequel. Let $\mathbb{R}^l$ be the $l$-dimensional Euclidean space. For $a, b\in\mathbb{R}^l$, by $a\leqq b$, we mean $a_i\leq b_i$ for all $i=1, \ldots, l$; by $a\leq b$, we mean $a\leqq b$ and $a\neq b$; and by $a<b$, we mean $a_i<b_i$ for all $i=1, \ldots, l$. For any two vectors $a=(a_1, a_2)$ and $b=(b_1, b_2)$ in $\mathbb{R}^2$, we denote the lexicographic order by
\begin{align*}
a&\leqq_{\rm lex} b,\ \  {\rm iff} \ \ a_1<b_1\ \   {\rm or} \ \  a_1=b_1\ \ {\rm and }\ \   a_2\leq b_2,
\\
a&<_{\rm lex} b,\ \  {\rm iff} \ \ a_1<b_1\ \   {\rm or} \ \  a_1=b_1\ \ {\rm and }\ \   a_2< b_2.
\end{align*}

Fix $x^0\in Q_0$, the {\em active index set} at $x^0$ is defined by
$$J(x^0):=\{j\in J\,:\,g_j(x^0)=0\}.$$
For each $u\in\mathbb{R}^n$, put  $J(x^0; u):=\{j\in J(x^0)\,:\, \langle \nabla g_j(x^0), u\rangle=0\}.$ We say that $u$ is a {\em critical direction} of problem \eqref{problem}  at $x^0\in Q_0$   iff 
\begin{align*}
\langle \nabla f_i(x^0), u\rangle&= 0, \ \ \ \forall i\in I,
\\
\langle \nabla g_j(x^0), u\rangle&\leq 0, \ \ \ \forall j\in J(x^0).
\end{align*}
The set of all critical direction of the problem \eqref{problem} at $x^0$ is denoted by $K(x^0)$. The following sets were introduced by Maeda \cite{Maeda04}:
\begin{align*}
Q^i&:=Q_0\cap\{x\in \mathbb{R}^n\,:\, f_k(x)\leq f_k(x^0), k\in I\setminus\{i\}\}, \ \ i \in I,\\
Q&:=Q_0\cap\{x\in \mathbb{R}^n\,:\, f_k(x)\leq f_k(x^0), k\in I\}.
\end{align*}
If $I=\{1\}$,  we set $Q^i:=Q_0$. The  sets  were proposed by Rizvi and Nasser \cite{Rizvi06}.
$$
M^i:=Q_0\cap\{x\in \mathbb{R}^n\,:\, f_i(x)\leq f_i(x^0)\},\ \ \ i\in I.
$$
Clearly, $\displaystyle Q^i=\mathop{\bigcap\limits_{k\in I}}\limits_{k\neq i} M^k$ for all $i\in I$ and $\displaystyle Q=\bigcap_{i\in I}Q^i=\bigcap_{i\in I}M^i.$
\begin{remark}{\rm
	By the isotony of $T(\cdot\,; x^0)$ and $T^2(\cdot\,; x^0, u)$, we have
	$$T(Q^i; x^0)\subset T(M^k; x^0), 	T^2(Q^i; x^0, u)\subset T^2(M^k; x^0, u),\ \ \ \forall k\in I\setminus \{i\}.$$
	Thus
	\begin{equation*}
	\bigcap_{i\in I}T(Q^i; x^0)\subset \bigcap_{i\in I}T(M^i; x^0)\ \ \mbox{and}\ \
	\bigcap_{i\in I}T^2(Q^i; x^0, u)\subset \bigcap_{i\in I} T^2(M^i; x^0, u). 	\end{equation*}	}
\end{remark}

The following example shows that the above inclusions may be strictly proper.
\begin{example}\label{strict_inclusion}{\rm Let $f\colon \mathbb{R}^2 \to\mathbb{R}^3$ and $g\colon \mathbb{R}^2\to \mathbb{R}$ be two maps defined by
	\begin{align*}
	f(x)&:= (f_1(x), f_2(x), f_3(x))=(x_2, x_1+x_2^2, -x_1-x_1|x_1|+x_2^2)\\
	g(x)&:=0, \ \ \forall x=(x_1, x_2)\in\mathbb{R}^2. 
	\end{align*}
	Clearly, $Q_0=\mathbb{R}^2$ and $x^0=(0,0)$ is a feasible point to problem \eqref{problem}. We have
	\begin{align*}
	&M^1=\{(x_1, x_2)\,:\, x_2\leq 0\}, M^2=\{(x_1, x_2)\,:\, x_1+x_2^2\leq 0\}, 
	\\
	&M^3=\{(x_1, x_2)\,:\, -x_1-x_1|x_1|+x_2^2\leq 0\},
	\\
	&Q^1=\{x^0\}, Q^2=M^1\cap M^3, Q^3=M^1\cap M^2.
	\end{align*}
	An easy computation shows that
	\begin{align*}
	&T(M^1; x^0)=M^1, T(M^2; x^0)=\{(u_1, u_2) : u_1\leq 0\}, 	 
	\\
	&T(M^3; x^0)= \{(u_1, u_2)\,:\, u_1\geq 0\}, \bigcap_{i=1}^3 T(M^i; x^0)=\{(u_1, u_2)\in\mathbb{R}^2:u_1=0, u_2\leq 0\}. 
	\end{align*}
	Clearly, $T(Q^1; x^0)=\{x^0\}$. Consequently, $\displaystyle\bigcap_{i=1}^3 T(Q^i; x^0)=\{x^0\}.$
	Thus 
	$$\bigcap_{i=1}^3 T(Q^i; x^0)\subsetneqq \bigcap_{i=1}^3 T(M^i; x^0).$$
	From this and the fact that 
	$$T^2(M^i; x^0, 0)=T(M^i; x^0), T^2(Q^i; x^0, 0)=T(Q^i; x^0),$$
	for all $i=1, 2, 3$, we have 
	$$\bigcap_{i=1}^3T^2(Q^i; x^0, 0)\subsetneqq \bigcap_{i=1}^3 T^2(M^i; x^0, 0).$$}
\end{example}

Now we define the first-order and the second-order linear approximation sets to the set $Q$. Fix $u\in\mathbb{R}^n$. By the compactness of $\partial_S^2f_i(x^0)(u)$ and $\partial_S^2g_j(x^0)(u)$, {\em we always denote hereafter} that $\xi^{*i}$ and $\zeta^{*j}$ are elements  in  $\partial_S^2f_i(x^0)(u)$ and  $\partial_S^2g_j(x^0)(u)$, respectively, such that    
\begin{align*}
\langle\xi^{*i}, u\rangle&:=\max \left\{\langle \xi^i, u\rangle\,:\, \xi^i\in \partial_S^2f_i(x^0)(u)\right\},\ \ i\in I,
\\
\langle\zeta^{*j}, u\rangle&:=\max \left\{\langle \zeta^j, u\rangle\,:\,\zeta^j\in \partial_S^2g_j(x^0)(u) \right\},\ \ j\in J.
\end{align*}
For each $(u, v)\in\mathbb{R}^n\times \mathbb{R}^n$, put
\begin{align*} 
F^2_i(x^0; u, v)&:=\left(\langle \nabla f_i(x^0), u\rangle, \langle \nabla f_i(x^0), v\rangle+\langle\xi^{*i}, u\rangle\right),\ \  i\in I,\\
G^2_j(x^0; u, v)&:=\left(\langle \nabla g_j(x^0), u\rangle, \langle \nabla g_j(x^0), v\rangle+\langle\zeta^{*j}, u\rangle\right),\ \ j\in J.
\end{align*}
\begin{definition}{\rm Let $x^0\in Q_0$ and $u\in\mathbb{R}^n$.
	\begin{enumerate}
		\item [(i)] The {\em linearizing cone} to $Q$ at $x^0\in Q_0$ is defined by
		$$L(Q; x^0):= \{u\in\mathbb{R}^n\,:\, \langle\nabla f_i (x^0), u\rangle \leq 0, \langle\nabla g_j (x^0), u\rangle \leq 0, i\in I, j\in J(x^0)\}.$$
		\item [(ii)] The {\em second-order linearizing set} of $Q$ at $x^0\in Q_0$ in the direction $u$ is defined by
		\begin{eqnarray*}
			L^2(Q; x^0, u)&&:=\{v\in\mathbb{R}^n\,\, :\, F^2_i(x^0; u, v)\leqq_{\rm lex} (0, 0),\ \  i\in I,
			\\
			&& \qquad \qquad \text{and} \ \ G^2_j(x^0; u, v)\leqq_{\rm lex} (0,0),\ \ j\in J(x^0)\}.
		\end{eqnarray*}
	\end{enumerate}	
}
\end{definition}

Then, we introduce two types of second-order  regularity conditions in the sense of Abadie as follows. 
\begin{definition}{\rm Let $x^0\in Q_0$ and $u\in\mathbb{R}^n$. We say that:
	\begin{enumerate}
		\item [(i)] The {\em  Abadie second-oder  regularity condition} holds at $x^0$ for the direction $u$ iff
		\[
		L^{2}(Q; x^0, u)\subset \bigcap_{i=1}^l T^2(Q^i; x^0, u). \tag{$ASORC$}\label{SARC}
		\]
		\item [(ii)] The {\em generalized Abadie second-oder  regularity condition} holds at $x^0$ for the direction $u$ iff 
		\[
		L^{2}(Q; x^0, u)\subset \bigcap_{i=1}^l T^2(M^i; x^0, u). \tag{$GASORC$}\label{GSARC}
		\]
	\end{enumerate}
}
\end{definition}
\begin{definition} Let  $x^0\in Q_0$. We say that the {\em Abadie first-order  regularity condition $(AFORC)$} (resp., {\em generalized Abadie first-oder  regularity condition $(GAFORC)$})    holds at $x^0$ iff \eqref{SARC} (resp., \eqref{GSARC}) holds at $x^0$ for the direction $0$.	
\end{definition}
\begin{remark}\label{relation_SORC}{\rm
	\begin{enumerate}
		\item [(i)] From the isotony of  second-order tangent sets, if the \eqref{SARC} holds at $x^0$ for the direction $u$, then so does the \eqref{GSARC}. The reverse does not hold in general; see Example \ref{not_Geoffrion} in Sect. \ref{Strong_KT_section}. 
		\item [(ii)] When $f$ and $g$ are $C^2$ functions, then the \eqref{GSARC} coincides with the generalized Abadie second-oder  regularity condition in the sense of Rizvi and Nasser  \cite{Rizvi06}.
		\item [(iii)] Let $\varphi\in C^{1,1}$. By \cite[Proposition 2.1.2]{Clarke83}, we have  
		$$\varphi^{\circ \circ}(x^0;u,u)=\max\,\{\langle\xi,u \rangle\,:\, \xi\in \partial_C \langle\nabla \varphi(\cdot), u\rangle (x^0)\}=\langle \xi^{*}, u\rangle,$$
		where
		\begin{align*}
		\varphi^{\circ \circ}(x^0;u,u)&:=\limsup\limits_{\mathop {x \to x^0}\limits_{t  \downarrow 0} } \frac{\langle\nabla \varphi(x+tu), u\rangle-\langle\nabla \varphi(x), u\rangle}{t},
		\\
		\langle \xi^{*}, u\rangle&:=\max \left\{\langle \xi, u\rangle\,:\, \xi \in \partial_S^2\varphi(x^0)(u)\right\}.
		\end{align*}
		Thus, the \eqref{SARC}  coincides with the  Abadie second-oder  regularity condition in the sense of Maeda \cite{Maeda04}.
	\end{enumerate}}
\end{remark}
We now introduce a condition, which ensures that the  \eqref{GSARC} holds at $x^0$ for the direction $u$.

\begin{proposition}\label{sufficient_GSARC} Let $x^0\in Q_0$ and  $u\in\mathbb{R}^n$. Suppose that, for each $i\in I$, the following system  (in the unknown $v$)
	\begin{eqnarray}
	\langle \nabla f_i(x^0), v\rangle+  \langle \xi^{*i}, u\rangle&<0,&   \label{equ:14}
	\\
	\langle \nabla g_j(x^0), v\rangle+ \langle \zeta^{*j}, u\rangle&<0,& \ \ \  j\in J(x^0; u), \label{equ:15}
	\end{eqnarray}
	has at least one solution, say $v^i\in\mathbb{R}^n$. Then, the \eqref{GSARC} holds at $x^0$ for the direction $u$. 
\end{proposition}
\begin{proof} Let $v$ be an arbitrary element in $L^2(Q; x^0, u)$. Fix $i\in I$ and let $v^i$ be a solution of the system \eqref{equ:14}-\eqref{equ:15}.     Let $\{s_h\}$ and $\{t_p\}$ be any positive sequences converging to 0. For each $h\in\mathbb{N}$, put $w^h:=s_hv^i+(1-s_h)v$.  Clearly, $\displaystyle\lim_{h\to\infty}w^h=v$. Since
	$v\in L^2(Q; x^0, u)$, we have
	\begin{equation}\label{equa:6}
	\begin{cases}
	F^2_k(x^0; u, v)&=\left(\langle \nabla f_k(x^0), u\rangle, \langle \nabla f_k(x^0), v\rangle+\langle\xi^{*k}, u\rangle\right)\leqq_{\rm lex} (0,0), \,\,k\in I,
	\\
	G^2_j(x^0; u, v)&=\left(\langle \nabla g_j(x^0), u\rangle, \langle \nabla g_j(x^0), v\rangle+\langle\zeta^{*j}, u\rangle\right)\leqq_{\rm lex} (0,0),\, \, j\in J(x^0).
	\end{cases}
	\end{equation}
	This implies that
	$$
	\begin{cases}
	\langle \nabla f_k(x^0), u\rangle\leq 0,\ \ \ \forall k\in I,
	\\
	\langle \nabla g_j(x^0), u\rangle\leq 0, \ \ \ \forall j\in J(x^0).
	\end{cases}
	$$
	For $h=1$, we have $w^1=s_1v^i+(1-s_1)v$. We consider two cases of $i$ as follows.
	\\
	{\bf Case 1.} $\langle \nabla f_i(x^0), u\rangle= 0$. By \eqref{equa:6}, we have $\langle \nabla f_i(x^0), v\rangle+\langle\xi^{*i}, u\rangle \leq 0.$
	Since $v^i$ is a solution of the system \eqref{equ:14}-\eqref{equ:15}, we have
	\begin{align}
	\langle\nabla f_i(x^0), w^1\rangle+\langle \xi^{*i}, u\rangle=&s_1[\langle\nabla f_i(x^0), v^i\rangle+\langle \xi^{*i}, u\rangle]\notag
	\\
	&+(1-s_1)[\langle\nabla f_i(x^0), v\rangle+\langle \xi^{*i}, u\rangle]<0. \label{equa:2}
	\end{align}
	For each $p\in \mathbb{N}$, put $x^p:=x^0+t_pu+\frac12t^2_pw^{1}$. Then we have
	\begin{equation*}
	f_i(x^p)-f_i(x^0)=[f_i(x^p) -f_i(x^0+t_pu)]+[f_i(x^0+t_pu)-f_i(x^0)-t_p\langle\nabla f_i(x^0),u\rangle].
	\end{equation*}
	From the mean value theorem for differentiable functions, we have
	\begin{equation}
	f_i(x^p)-f_i(x^0+t_pu)=\langle \nabla f_i(\gamma^s),  \frac12t_p^2w^1\rangle=\frac12t_p^2\langle \nabla f_i(\gamma^p),  w^1\rangle \label{equ:11}
	\end{equation}
	for some  $\gamma^p\in (x^0+t_pu, x^p)$. By Taylor's formula, there exist $\eta^p\in (x^0, x^0+t_pu)$ and $\zeta^s\in \partial_S^2f_i(\eta^p)(t_pu)$ such that
	$$f_i(x^0+t_pu)-f_i(x^0)-t_p\langle\nabla f_i(x^0), u\rangle =\frac12 \langle \zeta^p, t_pu\rangle= \frac12t_p \langle \zeta^p, u\rangle.$$
	Furthermore, since $\partial_S^2f_i(\eta^p)(t_pu)=t_p\partial_S^2f_i(\eta^p)(u)$,
	there is $\xi^p\in \partial_S^2f_i(\eta^p)(u)$ such that $\zeta^p=t_p\xi^p$. Thus 
	$$f_i(x^0+t_pu)-f_i(x^0)-t_p\langle\nabla f_i(x^0), u\rangle  =\frac12t^2_p\langle \xi^p, u\rangle.$$
	From this and \eqref{equ:11}, one has  
	\begin{equation}\label{equ:12}
	\dfrac{f_i(x^p)-f_i(x^0)}{\frac12t^2_p}=\langle\nabla f_i(\gamma^p),  w^1\rangle+\langle\xi^p, u\rangle.
	\end{equation}
	Since $\partial_S^2f_i(\cdot)(u)$ is locally bounded at $x^0$ and  $\displaystyle\lim_{p\to\infty}\eta^p=x^0$, it follows that the sequence $\{\xi^p\}$ is bounded. By the boundedness of $\{\xi^p\}$ and Proposition \ref{second-order_subdiff}(iii), without loss of any generality, we may assume that $\{\xi^p\}$ converges to $\xi^{i}\in \partial_S^2f_i(x^0)$. It is easily seen that
	$$\lim_{p\to\infty}\langle\nabla f_i(\gamma^p),  w^1\rangle=\langle\nabla f_i(x^0),  w^1\rangle.$$
	Letting $p\to\infty$ in \eqref{equ:12} we have
	\begin{equation}\label{equ:18}
	\lim_{p\to\infty}\dfrac{f_i(x^p)-f_i(x^0)}{\frac12t^2_p}=\langle\nabla f_i(x^0), w^{1}\rangle+\langle\xi^{i},u\rangle.
	\end{equation}
	Since \eqref{equa:2}, one has
	\begin{equation}\langle\nabla f_i(x^0), w^{1}\rangle+\langle\xi^{i},u\rangle\leq \langle\nabla f_i(x^0), w^1\rangle+\langle \xi^{*i}, u\rangle<0.\label{equa:4}
	\end{equation}
	By \eqref{equ:18} and \eqref{equa:4}, there exists $N_1\in\mathbb{N}$ such that 
	$f_i(x^p)-f_i(x^0)<0$ for all $p\geq N_1$,
	or, equivalently,
	\begin{equation}\label{equa:7}
	f_i(x^p)<f_i(x^0),\ \ \ \forall p\geq N_1.
	\end{equation}
	{\bf Case 2.} $\langle \nabla f_i(x^0), u\rangle< 0$. Since
	$$\lim_{p\to\infty}\dfrac{f_i(x^p)-f_i(x^0)}{t_p}=\langle\nabla f_i(x^0), u\rangle,$$
	there exists $N_2\in\mathbb{N}$ such that $f_i(x^p)<f_i(x^0)$ for all $p\geq N_2.$
	
	We now claim that  $g_j(x^p)<0$ for all $j\in J$ and  $p$ large enough. Indeed, we consider the following cases:
	\begin{enumerate}
		\item [(a)] $j\in J\setminus J(x^0)$. Then $g_j(x^0)<0$. Hence, by ${\displaystyle\lim_{p\to\infty}}x^p=x^0$ and the continuity of $g_j$,
		there exists $N_3\in\mathbb{N}$ such that $g_j(x^p)<0$ for all $p\geq N_3$.
		\item [(b)] $j\in J(x^0)\setminus J (x^0; u)$. Then $g_j(x^0)=0$ and $\langle \nabla g_j(x^0), u\rangle<0$.
		From this and 
		$$\lim_{p\to\infty}g_j(x^p)=\lim_{p\to\infty}\dfrac{g_j(x^p)-g_j(x^0)}{t_p}=\langle \nabla g_j(x^0), u\rangle$$
		it follows that there exists $N_4\in\mathbb{N}$ such that $g_j(x^p)<0$ for all $p\geq N_4.$
		\item [(c)]  $j\in J(x^0; u)$. An analysis similar to the one made in the proof of \eqref{equa:7} shows that there exist $N_5\in \mathbb{N}$ such that
		$g_j(x^p)<0$ for all $p\geq N_5$.
	\end{enumerate}
	
	Put $p_1=\max\{N_1, N_2, N_3, N_4, N_5\}$. Then we have
	\begin{align*}
	f_i(x^0+t_{p_1}u+\frac12t^2_{p_1}w^{1})&<f_i(x^0),\\
	g_j(x^0+t_{p_1}u+\frac12t^2_{p_1}w^{1})&<0,\ \ \ \forall j\in J.
	\end{align*}
	This means that  $x^0+t_{p_1}u+\frac12t^2_{p_1}w^{1}\in M^i.$
	Thus, by induction, we could construct a subsequence $\{t_{p_h}\}$ of $\{t_p\}$ satisfying $x^0+t_{p_h}u+\frac12t^2_{p_h}w^{h}\in M^i$ for all $h\in\mathbb{N}$.  From this and $\displaystyle\lim_{h\to\infty}w^h=v$ it follows that $v\in T^2(M^i; x^0, u)$. 
	Since $i$ is arbitrary, we have $v\in T^2(M^i; x^0, u)$ for all $i\in I$. Thus, the \eqref{GSARC} holds at $x^0$ for the direction $u$. 
\end{proof}
By a similar argument, we have the following result.
\begin{proposition} Let $x^0\in Q_0$ and  $u\in\mathbb{R}^n$. Suppose that, for each $i\in I$, the following system (in the unknown $w$)
	\begin{eqnarray*}
		\langle \nabla f_k(x^0), w\rangle+  \langle \xi^{*k}, u\rangle&<0,\ \ \ & k\in I\setminus\{i\},
		\\
		\langle \nabla g_j(x^0), w\rangle+ \langle \zeta^{*j}, u\rangle&<0, \ \ \ & j\in J(x^0; u),
	\end{eqnarray*}
	has at least one solution, say $w^i\in\mathbb{R}^n$. Then, the \eqref{SARC} holds at $x^0$ for the direction $u$. 
\end{proposition}
\section{Strong Second-Order  Karush--Kuhn--Tucker Necessary Optimality Conditions}\label{Strong_KT_section}
In this section, we provide some strong second-order Karush--Kuhn-Tucker necessary optimality conditions for Geoffrion properly efficient solutions of   \eqref{problem}.  Properly efficient solution plays a vital role from both theoretical and practical points of view. This concept has been introduced at first to eliminate the efficient solutions with unbounded trade-offs. In multiobjective optimization, properly efficient solutions are efficient solutions in which, given any objective, the trade-off between that objective and some other objective is bounded. This notion was introduced initially by Kuhn and Tucker \cite{Kuhn52} and was followed thereafter by Geoffrion
\cite{Geoffrion68}. Geoffrion's definition enjoys economical interpretations, while Kuhn and Tucker's definition is useful for numerical and algorithmic purposes. We now recall the definition of Geoffrion properly efficient solutions from \cite{Geoffrion68}.

\begin{definition} {\rm
	Let $x^0\in Q_0$. We say that:
	\begin{enumerate}
		\item [(i)] $x^0$ is an {\em efficient solution} to \eqref{problem} iff there is no $x\in Q_0$ satisfies  $f(x)\leq f(x^0)$.
		\item [(ii)] $x^0$ is a {\em Geoffrion properly efficient solution} to \eqref{problem} iff it is efficient and there exists $M>0$ such that, for each $i$, 
		$$\frac{f_i(x)-f_i(x^0)}{f_j(x^0)-f_j(x)}\leq M,$$
		for some $j$ such that $f_j(x^0)<f_j(x)$ whenever $x\in Q_0$ and $f_i(x^0)>f_i(x)$.	
	\end{enumerate} }
\end{definition}

The following result gives  a first-order necessary optimality condition for  \eqref{problem}  under the $(GAFORC)$.
\begin{theorem}\label{first_order-nec} {\rm (see \cite[Theorem 4.3]{Rizvi12})} If $x^0\in Q_0$ is a Geoffrion properly efficient solution to problem \eqref{problem} and the $(GAFORC)$ holds at $x^0$, then the following system has no solution $u\in\mathbb{R}^n$:
	\begin{eqnarray*}
		\langle \nabla f_i(x^0), u\rangle&\leq 0, \ \ \ &i\in I,  
		\\
		\langle \nabla f_i(x^0), u\rangle&< 0, \ \ \ &\mbox{at least one}\ \ i\in I,
		\\
		\langle \nabla g_j(x^0), u\rangle&\leq 0, \ \ \ &j\in J (x^0).  
	\end{eqnarray*}
\end{theorem}
\begin{definition}{\rm We say that the {\em strong second-order Krush--Kuhn-Tucker necessary optimality conditions} $(SSKKT)$ holds at $x^0$ for the direction $u$  iff  there exist  $\lambda\in \mathbb{R}^l$ and $\mu\in\mathbb{R}^m$  satisfying
	\begin{align}
	&\sum_{i=1}^l\lambda_i\nabla f_i(x^0)+\sum_{j=1}^m\mu_j\nabla g_j(x^0)=0, \label{equa:22new1} 
	\\
	&\sum_{i=1}^l\lambda_i\langle \xi^{*i}, u\rangle+\sum_{j=1}^m\mu_j\langle \zeta^{*j}, u\rangle\geq 0, \label{equa:22new2} 
	\\
	&\mu=(\mu_1, \mu_2, \ldots, \mu_m)\geqq 0, \mu_j=0,\ \  j\notin J(x^0; u), \label{equa:22new4}
	\\
	&\lambda=(\lambda_1, \lambda_2, \ldots, \lambda_l)> 0.  \label{equa:22new3}
	\end{align}   }
\end{definition}

The following theorem is crucial for establishing the $(SSKKT)$.
\begin{theorem}\label{strong_KKT_I}  Let $x^0$ be a Geoffrion  properly efficient solution  to problem \eqref{problem}. Suppose that the \eqref{GSARC} holds at $x^0$ for any critical direction at $x^0$. Then, the following system 
	\begin{align}
	F^2_i(x^0; u, v)&\leq_{\rm lex} (0, 0),\ \ \ i\in I, \label{equ:20}
	\\
	F^2_i(x^0; u, v)&<_{\rm lex} (0, 0),\ \ \ \mbox{ at least one} \ \ i\in I, \label{equ:21}
	\\
	G^2_j(x^0; u, v)&\leqq_{\rm lex} (0, 0),\ \ \  j\in J(x^0) \label{equ:22}
	\end{align}
	has no solution $(u, v)\in\mathbb{R}^n\times\mathbb{R}^n$.
\end{theorem}
\begin{proof} Arguing by contradiction, assume that the system \eqref{equ:20}--\eqref{equ:22} admits one solution $(u, v)\in\mathbb{R}^n\times\mathbb{R}^n$. Without loss of generality we may assume that
	\begin{equation}\label{equa:21_new1}
	F^2_1(x^0; u, v)<_{\rm lex} (0, 0).
	\end{equation}
	From \eqref{equ:20} and \eqref{equ:22} it follows that $v\in L^2 (Q; x^0, u)$ and
	\begin{eqnarray*}
		\langle \nabla f_i(x^0), u\rangle&\leq 0, \ \ \ &i\in I,
		\\
		\langle \nabla g_j(x^0), u\rangle&\leq 0, \ \ \ &j\in J (x^0).
	\end{eqnarray*}
	Since the \eqref{GSARC} holds at $x^0$ for any critical direction, the $(GAFORC)$ holds at $x^0$.    By Theorem \ref{first_order-nec}, we have 
	$\langle \nabla f_i(x^0), u\rangle=0$ for all $i\in I$. Thus, $u$ is a critical direction at $x^0$. Since the \eqref{GSARC} holds at $x^0$ for the critical direction $u$, we have $v\in T^2(M^i; x^0, u)$ for all $i\in I$. Consequently, $v\in T^2(M^1; x^0, u)$. This implies that there exist a sequence  $\{v^k\}$ converging to $v$ and a positive sequence  $\{t_k\}$ converging to $0$ such that
	$$x^k:=x^0+t_ku+\frac12t_k^2v^k\in M^1,\ \ \ \forall k\in\mathbb{N}.$$ 
	Since $M^1\subset Q_0$, $\{x^k\}\subset Q_0$. By \eqref{equ:20} and \eqref{equa:21_new1}, we have
	\begin{align}
	\langle \nabla f_1(x^0), v\rangle +\langle \xi^{*1}, u\rangle&<0, \label{equa:21_new2} 
	\\
	\langle \nabla f_i(x^0), v\rangle +\langle \xi^{*i}, u\rangle&\leq 0,\ \ \ \forall i\in \{2, \ldots, l\}. \label{equa:21_new22}
	\end{align}
	An analysis similar to the one made in the proof of Proposition \ref{sufficient_GSARC} shows that, for each $i\in I$,  there exists $\xi^i\in \partial_S^2f_i(x^0)(u)$ such that 
	\begin{equation}\label{equa:21_new33}
	\lim_{k\to\infty}\dfrac{f_i(x^k)-f_i(x^0)}{\frac12t^2_k}=\langle\nabla f_i(x^0), v\rangle+\langle\xi^i, u\rangle.
	\end{equation}
	In particular,
	\begin{equation}\label{equa:21_new3}
	\lim_{k\to\infty}\dfrac{f_1(x^k)-f_1(x^0)}{\frac12t^2_k}=\langle\nabla f_1(x^0), v\rangle+\langle\xi^1, u\rangle.
	\end{equation}
	Since \eqref{equa:21_new2}, we have
	\begin{equation}\label{equa:21_new4}
	\langle\nabla f_1(x^0), v\rangle+\langle\xi^1, u\rangle<0.
	\end{equation}
	This and \eqref{equa:21_new3} imply that  
	$$f_1(x^k)-f_1(x^0)<0$$ 
	for all large enough $k$. Without loss of generality we may assume that 
	\begin{equation*}
	f_1(x^k)<f_1(x^0),\ \ \ \forall k\in\mathbb{N}.
	\end{equation*}
	For each $k\in \mathbb{N}$, put 
	$$I_k:= \{i\in I \ : \ i\geq 2\ \ \mbox{and} \ \ f_i(x^k)>f_i(x^0)\}.$$
	We claim that $I_k$ is nonempty for all $k\in\mathbb{N}$. Indeed, if $I_k=\emptyset$ for some $k\in \mathbb{N}$, then we have
	$$f_i(x^k)\leq f_i(x^0),\ \ \ \forall i=2, \ldots, l.$$
	Using also the fact that  $f_1(x^k)<f_1(x^0)$, we arrive at a contradiction with the efficiency of $x^0$.
	
	Since $I_k\subset \{2, \ldots, l\}$ for all $k\in\mathbb{N}$, without loss of generality, we may assume  that $I_k=\bar I$ is constant for all   $k\in \mathbb{N}$. By \eqref{equa:21_new33}, for each $i\in \bar I$, we have
	$$\langle\nabla f_i(x^0), v\rangle+\langle\xi^i, u\rangle\geq 0.$$
	Thus
	$$\langle\nabla f_i(x^0), v\rangle+\langle\xi^{*i}, u\rangle\geq \langle\nabla f_i(x^0), v\rangle+\langle\xi^i, u\rangle\geq 0.$$
	This and \eqref{equa:21_new22} imply that
	\begin{equation}\label{equa:21_new7}
	\langle\nabla f_i(x^0), v\rangle+\langle\xi^i, u\rangle=\langle\nabla f_i(x^0), v\rangle+\langle\xi^{*i}, u\rangle=0, 
	\end{equation}
	for all $i\in \bar I$. By \eqref{equa:21_new4}, we can choose $\delta\in \mathbb{R}$ such that
	\begin{equation*}
	\langle\nabla f_1(x^0), v\rangle+\langle\xi^1, u\rangle<\delta<0,
	\end{equation*}
	or, equivalently,
	\begin{equation*} 
	-[\langle\nabla f_1(x^0), v\rangle+\langle\xi^1, u\rangle]>-\delta>0.
	\end{equation*}
	From this and \eqref{equa:21_new3} it follows that there exists $k_0\in \mathbb{N}$ such that
	\begin{equation*}
	f_1(x^0)- f_1(x^k)>-\frac12\delta t^2_k>0,
	\end{equation*}
	for all $k\geq k_0$. Thus, for any $i\in \bar I$ and $k\geq k_0$, we have
	\begin{equation*}
	0< \dfrac{f_i(x^k)-f_i(x^0)}{f_1(x^0)- f_1(x^k)}\leq \dfrac{f_i(x^k)-f_i(x^0)}{-\frac12\delta t^2_k}. 
	\end{equation*}
	Combining this, \eqref{equa:21_new33} and \eqref{equa:21_new7}, we  deduce
	\begin{align*}
	0\leq \lim_{k\to\infty}\dfrac{f_i(x^k)-f_i(x^0)}{f_1(x^0)- f_1(x^k)}&\leq \lim_{k\to\infty}\dfrac{f_i(x^k)-f_i(x^0)}{-\frac12\delta t^2_k}
	\\
	&=-\frac{1}{\delta}[\langle\nabla f_i(x^0), v\rangle+\langle\xi^i, u\rangle]=0. 
	\end{align*}
	Thus
	$$\lim_{k\to\infty}\dfrac{f_1(x^k)-f_1(x^0)}{f_i(x^0)-f_i(x^k)}=+\infty,$$
	which contradicts that $x^0$ is a Geoffrion properly efficient solution to \eqref{problem}. 
\end{proof}

The following theorem gives some strong second-order Karush--Kuhn-Tucker necessary optimality conditions for Geoffrion properly efficient solutions of  problem \eqref{problem}. This result generalizes \cite[Theorem 3.2]{Kim17} to $C^{1,1}$ vector optimization problems. 
\begin{theorem}\label{KKT_necessary_II}Let $x^0\in Q_0$ be a Geoffrion properly efficient solution to problem \eqref{problem} and $u\in K(x^0)$. If   the \eqref{GSARC} holds at $x^0$ for the direction $u$, then so does  the $(SSKKT)$.
\end{theorem}
\begin{proof}  Suppose that $x^0$ is a Geoffrion properly efficient solution to problem \eqref{problem} and $u\in K(x^0)$. Thanks to Theorem \ref{strong_KKT_I}, the system 
	\begin{eqnarray*}
		\langle \nabla f_i(x^0), v\rangle+\langle \xi^{*i}, u\rangle&\leq 0,\ \ \ \ &i\in I,  
		\\
		\langle \nabla f_i(x^0), v\rangle+\langle \xi^{*i}, u\rangle&< 0,\ \ \ \ &\mbox{at least one } \ \ i\in I, 
		\\
		\langle \nabla g_j(x^0), v\rangle+ \langle \zeta^{*j}, u\rangle&\leq 0,\ \ \ \ &j\in J(x^0; u), 
	\end{eqnarray*}   
	has no solution $v\in\mathbb{R}^n$. This is equivalent to the inconsistency of the following system
	\begin{align}
	\langle \nabla f_i(x^0), v\rangle+\langle \xi^{*i}, u\rangle \epsilon&\leq 0,\ \ \ \ i=1, 2, ..., l, \label{equ:27}\\
	\langle \nabla f_i(x^0), v\rangle+\langle \xi^{*i}, u\rangle\epsilon&< 0 \ \ \ \ \text{at least one }\ \ i,\label{equ:28}\\
	\langle \nabla g_j(x^0), v\rangle+ \langle \zeta^{*j}, u\rangle\epsilon&\leq 0,\ \ \ \ j\in J(x^0; u),\label{equ:29}\\
	\epsilon&>0.\label{equ:29a}
	\end{align}
	Thus, by the Slater theorem \cite[p. 27]{Mangasarian69}, either the system
	\begin{align}
	&\sum_{i=1}^l\lambda_i\nabla f_i(x^0)+\sum_{j\in J(x^0; u)} \mu_j\nabla g_j(x^0)=0,\label{equ:30}\\
	&\sum_{i=1}^l\lambda_i\xi^{*i}+\sum_{j\in J}\mu_j\zeta^{*j}-\nu= 0,\label{equ:31}\\
	&\lambda_i> 0, i=1, 2, ..., l, \mu_j\geq 0, j\in J(x^0; u), \nu\geq 0, \label{equ:32}
	\end{align}
	has solution $\lambda_i, \mu_j, \nu\in\mathbb{R}$, or the system \eqref{equ:30}, \eqref{equ:31} and 
	\begin{equation}\label{equ:33}
	\lambda_i\geq 0, i=1, 2, ..., l, \mu_j\geq 0, j\in J(x^0; u), \nu\geq 0, ,
	\end{equation}
	has solution $\lambda_i, \mu_j, \nu\in\mathbb{R}$. We claim that the system \eqref{equ:30}--\eqref{equ:32} has solution  $\lambda_i, \mu_j, \nu\in\mathbb{R}$. On the contrary, suppose that the system  \eqref{equ:30}--\eqref{equ:32} has no solution. By the Tucker theorem \cite[p. 29]{Mangasarian69}, the system \eqref{equ:27}--\eqref{equ:29} and $\epsilon\geq 0$ has a solution $\bar v\in \mathbb{R}^n$ and $\bar \epsilon\geq 0$. It is easily  seen that $\bar \epsilon=0$. For each $t>0$, put $u(t)=\bar v+tu$. Since $(\bar v, 0)$ is a solution of the system \eqref{equ:27}--\eqref{equ:29}, for $t>0$ sufficiently large, we have
	\begin{align*}
	\langle \nabla f_i(x^0), u(t)\rangle&=\langle \nabla f_i(x^0), \bar v\rangle+t\langle \nabla f_i(x^0), u\rangle \leq 0, i=1, 2, ..., l, \\
	\langle \nabla f_i(x^0), u(t)\rangle&=\langle \nabla f_i(x^0), \bar v\rangle+t\langle \nabla f_i(x^0), u\rangle < 0, \ \ \text{at least one}\ \ i,\\
	\langle \nabla g_j(x^0), u(t)\rangle&=\langle \nabla g_j(x^0), \bar v\rangle+t\langle \nabla g_j(x^0), u\rangle \leq 0, j\in J(x^0; u),\\
	\langle \nabla g_j(x^0), u(t)\rangle&=\langle \nabla g_j(x^0), \bar v\rangle+t\langle \nabla g_j(x^0), u\rangle < 0, j\in J(x^0)\setminus J(x^0; u),
	\end{align*}
	by Theorem \ref{first_order-nec}, which contradicts the fact that $x^0$ is a Geoffrion properly efficient solution of $(P)$. 
\end{proof}

Let us illustrate Theorem \ref{KKT_necessary_II}.
\begin{example}{\rm
	Let $f\colon \mathbb{R}^2\to \mathbb{R}^2$ be defined by $f(x)=(f_1(x), f_2(x))$, where
	$$ f_1(x)=-f_2(x)=
	\begin{cases} 2x^2_1+x^2_2+x_2+x^2_1\sin(\ln|x_1|), &\text{ if  } x=(x_1, x_2)\in \mathbb{R}^2,  x_1\neq 0,
	\\
	x^2_2+x_2,  & \text{ if } x_1=0,
	\end{cases}
	$$
	and let $g_1\colon\mathbb{R}^2\to \mathbb{R}$ be given by $g_1(x)=-x_2$ for $x=(x_1, x_2)\in\mathbb{R}^2$. Let us consider problem \eqref{problem} with the objective function $f$ and the constraint set 
	$$Q_0=\{x\in \mathbb{R}^2\;:\; g_1(x)\leq 0\}.$$
	It is easy to check that $x^0:=(0, 0)$ is a Geoffrion properly efficient solution to problem \eqref{problem}. An easy computation shows that
	$$
	\nabla f_1(x)=
	\begin{cases}
	\left(4x_1+2x_1\sin (\ln |x_1|)+x_1\cos(\ln|x_1|), 2x_2+1\right),&\mbox{  if } x_1\neq 0\\
	(0, 2x_2+1),&\mbox{ if } x_1=0,
	\end{cases}
	$$
	$\nabla f_2(x)=-\nabla f_1(x)$, $\nabla g_1(x)=(0, -1)$, and $\nabla^2g_1(x)=(0,0)$ for $x\in\mathbb{R}^2$. Thus
	$$\nabla f_1(x^0)=-\nabla f_2(x^0)=-\nabla g_1(x^0)=(0, 1),$$
	and $K(x^0)=\{(u_1, u_2)\in\mathbb{R}^2\;:\; u_2=0\}$. Let $u=(1, 0)$ be a critical direction at $x^0$. Then, we have
	$$
	\partial_S^2f_1(x^0)(u)=\partial_S\langle u, \nabla f_1(\cdot)\rangle(x^0)
	=\partial_S(\nabla_{x_1} f_1(\cdot))(x^0)
	=\partial_S(\nabla_{x_1} f_1(\cdot))(x^0).
	$$
	Since $\nabla_{x_1}f_1(\cdot)$ does not depend on $x_2$, we have $\partial_S(\nabla_{x_1} f_1(\cdot))(x^0)=(\partial_S \varphi(0), 0)$, where $\varphi(x_1):=\nabla_{x_1} f_1(x)$ for all $x=(x_1, x_2)\in\mathbb{R}^2$. Thus $\partial_S^2f_1(x^0)(u)=(\partial_S \varphi(0), 0)$.
	Thanks to \cite[Corollary 2.3]{Borw94}, one has
	$$
	\partial_S\varphi(0)=\partial_C\varphi(0)=\left[\liminf_{x_1\to 0, x_1\neq 0} \nabla \varphi(x_1), \limsup_{x_1\to 0, x_1\neq 0}\nabla \varphi(x_1)\right].
	$$
	From
	$\varphi(x_1)=\nabla_{x_1} f_1(x)=4x_1+2x_1\sin (\ln |x_1|)+x_1\cos(\ln|x_1|)$  
	it follows that
	$\nabla \varphi(x_1)=4+\sin(\ln|x_1|)+3\cos(\ln|x_1|)$ for all $x_1\neq 0$. It is easy to check that
	$$
	\liminf_{x_1\to 0, x_1\neq 0} \nabla \varphi(x_1) =4-\sqrt{10} \ \ \mbox{and} \ \ \limsup_{x_1\to 0, x_1\neq 0}\nabla \varphi(x_1)=4+\sqrt{10}.
	$$
	Consequently,
	$$\partial_S^2f_1(x^0)(u)=[4-\sqrt{10}, 4+\sqrt{10}]\times\{0\},$$
	and hence $\partial_S^2f_2(x^0)(u)=[-4-\sqrt{10}, -4+\sqrt{10}]\times\{0\}$. It follows that
	$$\langle\xi^{*1}, u\rangle=\max \left\{\langle \xi^1, u\rangle\,:\, \xi^1\in \partial_S^2f_1(x^0)(u)\right\}=4+\sqrt{10},$$
	and 
	$$\langle\xi^{*2}, u\rangle=\max \left\{\langle \xi^2, u\rangle\,:\, \xi^2\in \partial_S^2f_2(x^0)(u)\right\}=-4+\sqrt{10}.$$
	Thus, for each $v=(v_1, v_2)\in \mathbb{R}^2$, we have
	$$F^2_1(x^0; u,v)=(0, v_2+4+\sqrt{10}), F^2_2(x^0; u,v)=(0, -v_2-4+\sqrt{10}), $$
	and $G_1^2(x^0; u,v)=(0, -v_2)$. It is easily seen that $L^2(Q; x^0, u)=\emptyset$. Thus, the \eqref{GSARC} holds at $x^0$. By Theorem \ref{KKT_necessary_II}, the $(SSKKT)$ holds at $x^0$ for the direction $u$. Moreover, it is easy to check that $(\lambda_1, \lambda_2)=(1, 1)$ and $\mu=0$ satisfy the conditions \eqref{equa:22new1}--\eqref{equa:22new3}. Since the functions $f_1$ and $f_2$ are not twice differentiable at $x^0$, Theorem 3.2 in \cite{Kim17} cannot be employed.}
\end{example}
\begin{remark}{\rm
	\begin{enumerate}
		\item [(i)] It is worth noting that Theorem \ref{KKT_necessary_II}  embraces also the $(SFKKT)$ optimality conditions given by Burachik and Rizvi \cite[Theorem 4.4]{Rizvi12}, which can be obtained by just considering the particular case $u=0$.
		
		\item [(ii)] Due to Remark \ref{relation_SORC}, the conclusions of Theorem \ref{strong_KKT_I} and Theorem \ref{KKT_necessary_II}  still hold when the assumption ``{\em the \eqref{GSARC} holds at $x^0$}'' is replaced by ``{\em the \eqref{SARC} holds at $x^0$}''. 
		\item [(iii)] Theorem \ref{strong_KKT_I} and Theorem \ref{KKT_necessary_II} cannot be extended to the efficient solution of \eqref{problem}. To see this, let us consider the following example.
	\end{enumerate}  }
\end{remark}

\begin{example}\label{not_Geoffrion}{\rm Consider problem \eqref{problem} with the objective function $f$ from Example \ref{strict_inclusion} and the constraint set $Q_0:=\mathbb{R}^2$. It is easy to check that $x^0:=(0, 0)$ is an efficient solution to problem $\rm(VP)$ and $Q=\{x^0\}$.   Since
	$$\nabla f_1(x^0)=(0, 1), \nabla f_2(x^0)=(1, 0), \nabla f_3(x^0)=(-1, 0),$$
	we have
	\begin{equation*}
	L^2(Q; x^0, 0_{\mathbb{R}^2})=L(Q; x^0)=\{(u_1, u_2)\,:\, u_1=0,  u_2\leq 0\}.
	\end{equation*}
	Thanks to Example \ref{strict_inclusion}, we have $K(x^0)=\{0_{\mathbb{R}^2}\}$ and
	\begin{align*}
	&L^2(Q; x^0, 0_{\mathbb{R}^2})=L(Q^0; x^0)=\bigcap_{i=1}^3 T(M^i; x^0)=\bigcap_{i=1}^3 T^2(M^i; x^0, 0_{\mathbb{R}^2}),
	\\
	&L^2(Q; x^0, 0_{\mathbb{R}^2})=L(Q^0; x^0)\nsubseteq \bigcap_{i=1}^3 T(Q^i; x^0)= \bigcap_{i=1}^3 T^2(Q^i; x^0, 0_{\mathbb{R}^2}), 
	\end{align*}
	where $0_{\mathbb{R}^2}:=(0,0)$. Thus the \eqref{GSARC} holds at $x^0$ but not the \eqref{SARC}. Since
	$$\lambda_1\nabla f_1(x^0)+\lambda_2\nabla f_2(x^0)+\lambda_3\nabla f_3(x^0)=0 \Leftrightarrow 
	\begin{cases}
	\lambda_1=0
	\\
	\lambda_2=\lambda_3,
	\end{cases}
	$$ 
	it follows that the $(SSKKT)$ does not hold at $x^0$ for the direction $0_{\mathbb{R}^2}$. Thus Theorem \ref{KKT_necessary_II} does not hold for  $x^0$.  Furthermore, for each $v=(v_1, v_2)\in \mathbb{R}^2$, we have
	\begin{equation*}
	F_1^2(x^0; 0_{\mathbb{R}^2}, v)=(0, v_2), F_2^2(x^0; 0_{\mathbb{R}^2}, v)=(0, v_1), F_3^2(x^0; 0_{\mathbb{R}^2}, v)=(0, -v_1).
	\end{equation*}
	Clearly, the system 
	\begin{equation*}
	\begin{cases}
	F_1^2(x^0; 0_{\mathbb{R}^2}, v) = (0, v_2) &<_{\rm lex} (0,0)
	\\
	F_2^2(x^0; 0_{\mathbb{R}^2}, v)= (0, v_1) &\leqq_{\rm lex} (0,0)
	\\
	F_3^2(x^0; 0_{\mathbb{R}^2}, v)=(0, -v_1)&\leqq_{\rm lex} (0,0),
	\end{cases}
	\end{equation*}
	admits a solution $v=(0, v_2)$, where $v_2<0$. This means that Theorem \ref{strong_KKT_I} does not hold for efficient solutions, too.
	
	We also note here that $x^0$ is not a Geoffrion properly efficient solution.  Indeed, let $x=(0, -a)$, $a>0$, then we have $f_1(x)<f_1(x^0), f_2(x)>f_2(x^0)$  and 
	\begin{equation*}
	\lim_{a\downarrow 0}\dfrac{f_1(x)-f_1(x^0)}{f_2(x^0)-f_2(x)}=\lim_{a\downarrow 0}\dfrac{1}{a}= +\infty.
	\end{equation*}
	Therefore $x^0$ is not a Geoffrion properly efficient solution.}
\end{example}
\section{Conclusions}
\label{Conclusion}
By using the second-order symmetric subdifferential, we introduce some types of  second-order  regularity conditions in the sense of Abadie and formulate strong second-order Karush--Kuhn--Tucker   necessary  optimality conditions for Geoffrion properly efficient solutions of $C^{1,1}$ vector optimization problems. It is meaningful that how to establish second-order Karush--Kuhn--Tucker-type optimality conditions for efficient solutions of  vector optimization problems purely in the Mordukhovich subdifferential terms. We aim to investigate this problem in future work.



\begin{thebibliography}{99}
	\bibitem{hien84}  Hiriart-Urruty JB,   Strodiot JJ, Nguyen VH. Generalized Hessian matrix and second-order optimality conditions for problems with $C^{1,1}$ data. { Appl. Math. Optim.} 1984;11:43--56.
	
	\bibitem{Jahn04} Jahn J. Vector optimization. New York (NY): Springer; 2004.
	
	\bibitem{Kuhn52} Kuhn HW, Tucker AW. Nonlinear programming. In: Neyman J, editor. Proceedings of the Second Berkeley Symposium on Mathematical Statistics and Probability; 1950 Jul 31-Aug 12; Berkeley (CA): University of California press; 1951. p. 481--492.
	
	\bibitem{Singh87} Singh, C. Optimality conditions in multiobjective differentiable programming. J. Optim. Theory Appl. 1987;53:115--123.
	
	\bibitem{Maeda94} Maeda T. Constraint qualification in multiobjective optimization problems: differentiable case. J. Optim. Theory Appl. 1994;80:483--500.
	
	\bibitem{Preda99} Preda  V, Chi{\c{t}}escu I.  On constraint qualification in multiobjective optimization problems: semidifferentiable case. J. Optim. Theory Appl. 1999;100:417--433.
	
	\bibitem{Giorgi 09} Giorgi G, Jim\'enez B, Novo V. Strong Kuhn-Tucker conditions and constraint qualifications in locally Lipschitz multiobjective optimization problem. Top. 2009;17:288--304.
	
	\bibitem{Golestani13} Golestani M, Nobakhtian S. Nonsmooth multiobjective programming: strong KuhnTucker conditions. Positivity. 2013;17:711--732.
	
	\bibitem{Rizvi12} Burachik RS, Rizvi MM.  On weak and strong Kuhn-Tucker conditions for smooth multiobjective optimization. J. Optim. Theory Appl. 2012;155:477--491. 
	
	\bibitem{Chuong14} Chuong TD, Yao JC. Isolated and proper efficiencies in semi-infinite vector optimization problems. J. Optim. Theory Appl. 2014;162:447--462.
	
	\bibitem{Wang91} Wang S. Second order necessary and sufficient conditions in multiobjective programming. { Numer. Funct. Anal. Optim.} 1991;12:237--252.
	
	\bibitem{Bigi00}  Bigi G, Castellani M.  Second order optimality conditions for differentiable multiobjective problems. { RAIRO Oper. Res.} 2000;34:411--426.
	
	\bibitem{Bigi04} Bigi G, Castellani M. Uniqueness of KKT multipliers in multiobjective optimization. { Appl. Math. Lett.} 2004;17:1285--1290.
	
	\bibitem{Aghezzaf99} Aghezzaf B, Hachimi M. Second-order optimality conditions in multiobjective optimization problems. J. Optim. Theory Appl. 1999;102:37--50.
	
	\bibitem{Aghezzaf07} Aghezzaf B, Hachimi M. New results on second-order optimality conditions in vector optimization problems. J. Optim. Theory Appl. 2007;135:117--133.
	
	\bibitem{Guerraggio01} Guerraggio  A, Luc  DT, Minh  NB. Second-order optimality conditions for $C^1$ multiobjective programming problems. Acta Math. Vietnam. 2001;26:257--268.
	
	\bibitem{Khanh16} Khanh PQ, Tung NM. Second-order conditions for open-cone minimizers and firm minimizers in set-valued optimization subject to mixed constraints. {J. Optim. Theory Appl.} 2016;171:45--69.
	
	\bibitem{novo03} Jim\'enez  B,   Novo  V. Second order necessary conditions in set constrained differentiable vector optimization. { Math. Methods Oper. Res.} 2003;58:299--317.
	
	\bibitem{novo09} Guti\'errez C, Jim\'enez  B,   Novo  V. New second-order directional derivative and optimality conditions in scalar and vector optimization. {J. Optim. Theory Appl.} 2009;142:85--106.
	
	\bibitem{Tuan15}  Tuan  ND.  First and second-order optimality conditions for nonsmooth vector optimization using set-valued directional derivatives. Appl. Math. Comput. 2015;251:300--317.
	
	\bibitem{Tuan16} Tuan ND. On necessary optimality conditions for nonsmooth
	vector optimization problems with mixed constraints in infinite dimensions. { Appl. Math. Optim.} 2016. DOI:10.1007/s00245-016-9383-z 
	
	\bibitem{Maeda04} Maeda T. Second-order conditions for efficiency in nonsmooth multiobjective optimization. { J. Optim. Theory Appl.} 2004;122:521--538.
	
	\bibitem{Kim17} Kim DS, Tuyen NV. A note on  second-order Karush--Kuhn--Tucker necessary   optimality conditions for smooth vector optimization problems. DOI:10.1051/ro/2017026
	
	\bibitem{Rizvi06} Rizvi  MM, Nasser M. New second-order optimality conditions in multiobjective optimization problems: differentiable case. J. Indian Inst. Sci. 2006;86:279--286.
	
	\bibitem{Huy16} Huy  NQ, Tuyen  NV. New second-order optimality conditions for a class of differentiable optimization problems. J. Optim. Theory Appl. 2016;171:27--44.
	
	\bibitem{Mor06a} Mordukhovich BS. Variational analysis and generalized differentiation, I: basic theory, 	Grundlehren series (fundamental principles of mathematical sciences). Vol. 330. Berlin: Springer; 2006.
	
	\bibitem{Clarke83} Clarke FH. Optimization and nonsmooth analysis. Philadelphia (PA): SIAM; 1990.
	
	\bibitem{Geoffrion68} Geoffrion  AM. Proper efficiency and the theory of vector maximization. J. Math. Anal. Appl. 1968;22:618--630.
	
	\bibitem{Mangasarian69} Mangasarian OL. Nonlinear programming. New York (NY): McGraw-Hill; 1969.
	
	\bibitem{Borw94} Borwein JM.  Fitzpatrick S. Characterization of Clarke subgradients among one-dimensional multifunctions. In: Glover BM, Jeyakumar V, editors.   Proceedings of the Optimization Miniconference II; 1994 Jul 14; Sydney (NSW): University of New South Wales; 1995. p. 61--64. 
	
	
\end{thebibliography}
\end{document}